\documentclass[12pt]{amsart}

\usepackage{amssymb}
\usepackage{amsfonts}
\usepackage{amsmath}
\usepackage{amscd}
\usepackage{bbm}
\usepackage{mathrsfs}

\newtheorem{theorem}{Theorem}[section]

\newtheorem{corollary}[theorem]{Corollary}
\newtheorem{lemma}[theorem]{Lemma}

\theoremstyle{definition}

\newtheorem{remark}[theorem]{Remark}

\newtheorem{definition}[theorem]{Definition}

\title{Divided difference operators in equivariant $KK$-theory}
\author{Ho-Hon Leung}
\date{}

\begin{document}

\begin{abstract}
Let $G$ be a compact connected Lie group with a maximal torus $T$. Let $A$, $B$ be $G$-$\mathrm{C}^\ast$-algebras. We define certain divided difference operators on Kasparov's $T$-equivariant $KK$-group $KK_T(A,B)$ and show that $KK_G(A,B)$ is a direct summand of $KK_T(A,B)$. More precisely, a $T$-equivariant $KK$-class is $G$-equivariant if and only if it is annihilated by an ideal of divided difference operators. This result is a generalization of work done by Atiyah, Harada, Landweber and Sjamaar.
\end{abstract}

\maketitle



\section{Introduction} \label{section1}	
Let $G$ be a compact connected Lie group, $T$ be a maximal torus of $G$ and $X$ be a compact $G$-space. In \cite{Atiyah}, Atiyah showed that $K_G^\ast(X)$ is a direct summand of $K_T^\ast(X)$. The restriction map from the $G$-equivariant $K$-ring $K_G^\ast(X)$ to the $T$-equivariant $K$-ring $K_T^\ast(X)$ has a natural left inverse. This pushforward homomorphism is defined by means of the Dolbeault operator associated with an invariant complex structure on the homogeneous space $G/T$. In \cite{Sjamaar}, Harada, Landweber and Sjamaar showed that the action of the Weyl group $W$ on $K_T^\ast(X)$ extends to an action of a Hecke ring $\mathscr{D}$ generated by \emph{divided difference operators}, which was first introduced in the context of Schubert calculus by Demazure \cite{Demazure1}, \cite{Demazure2} and \cite{Demazure}. The ring $\mathscr{D}$ contains an \emph{augmentation left ideal} $I(\mathscr{D})$ and they showed that $K_G^\ast(X)$ is isomorphic to the subring of $K_T^\ast(X)$ annihilated by $I(\mathscr{D})$.

This paper can be seen as a natural generalization of these results from equivariant $K$-theory to equivariant $KK$-theory introduced by Kasparov \cite{Kasparov1}, \cite{Kasparov2}. First, we extend the action of the ring $\mathscr{D}$ to the Kasparov's $T$-equivariant $KK$-group $KK_T(A,B)$ where $A$ and $B$ are $G$-$\mathrm{C}^\ast$-algebras. Next, we show that $KK_G(A,B)$ is isomorphic to $KK_T(A,B)$ annihilated by $I(\mathscr{D})$. These are the main results in Section \ref{section2}. The key results of this paper rely on theorems due to Wasserman \cite{Wasserman}. Since it is unpublished, we will prove Wasserman's theorems in Section \ref{Wassermansection} and Section \ref{Rosenbergsection}. 

\section{Main Results} \label{section2}
We will first briefly introduce the basic notations of equivariant $KK$-theory which was originally due to Kasparov. For a more detailed account of $KK$-theory, see \cite{JensenThomsen}, \cite{Kasparov1} and \cite{Kasparov2}.

A \emph{$G$-equivariant Kasparov $A-B$-module} is a triple $\mathcal{E}=(E,\phi,F)$ where $E$ is a countably generated graded $G$-Hilbert $B$-module, $A,B$ are $G$-$\mathrm{C}^\ast$-algebras, $\phi\colon A\rightarrow\mathrm{B}(E)$ is a graded $G$-$\ast$-homomorphism and $F\in\mathrm{B}(E)$ is an element of degree 1 such that for $a\in A$, $[F,\phi(a)]\in\mathrm{K}(E)$, $(F^2-1)\phi(a)\in\mathrm{K}(E)$, $(F^\ast-F)\phi(a)\in\mathrm{K}(E)$ and $(g.F-F)\phi(a)\in\mathrm{K}(E)$. The map $g\mapsto g.F-F$ is norm-continuous. $\mathrm{K}(E)$ is the set of $G$-equivariant compact operators acting on $E$. Note that the commutator $[F,\phi(a)]$ is graded and the continuous $G$-action on $E$ preserves the grading. We denote the set of $G$-equivariant Kasparov $A-B$-modules by $\mathbb{E}_G(A,B)$. 

We shall define $G$-equivariant $KK$-theory as follows. $\pi_t, t\in [0,1]$ is the surjection $IB=B\otimes C([0,1])\rightarrow B$ obtained by an evaluation at $t$. When $B$ is graded by an automorphism $\beta_B$, we consider $IB=B\otimes C([0,1])$ to be graded by the automorphism $\beta_B \otimes id$. $G$ acts on $B \otimes C([0,1])$ by $g.(b\otimes f)=(g.b)\otimes f$ for all $b\in B$, $g\in G$ and $f\in C([0,1])$. $\pi_t$ is then a graded $G$-$\ast$-homomorphism for all $t\in [0,1]$. If $\mathcal{E}\in\mathbb{E}_G(A,IB)$, we obtain a $G$-equivariant Kasparov $A-B$-module for each $t\in [0,1]$ by the \emph{pushout} $\mathcal{E}_{\pi_t}=[E_{\pi_t},\phi_{\pi_t},F_{\pi_t}]\in\mathbb{E}_G(A,B)$ where $E_{\pi_t}=E\otimes_{\pi_t}B$ is a $G$-Hilbert $B$-module by an internal tensor product of $E$ and $B$. $F_{\pi_t}$ and $\phi_{\pi_t}$ can then be defined accordingly.

Two $G$-equivariant Kasparov $A-B$-modules $\mathcal{E}_1=[E_1,\phi_1,F_1]$, $\mathcal{E}_2=[E_2,\phi_2,F_2]\in\mathbb{E}_G(A,B)$ are \emph{isomorphic} when there is a graded isomorphism $\psi\colon\mathcal{E}_1\rightarrow\mathcal{E}_2$ of $G$-Hilbert $B$-modules such that $F_2\circ \psi=\psi\circ F_1$ and $\phi_2(a)\circ\psi=\psi\circ\phi_1(a)$ for all $a\in A$. We write $\mathcal{E}_1\cong\mathcal{E}_2$ in this case.

Two $G$-equivariant Kasparov $A-B$ modules $\mathcal{E}, \mathcal{F}\in\mathbb{E}_G(A,B)$ are called \emph{homotopic} when there is a Kasparov $A-IB$-module $\mathcal{G}\in\mathbb{E}_G(A,IB)$ such that $\mathcal{G}_{\pi_0}\cong \mathcal{E}$ and $\mathcal{G}_{\pi_1}\cong\mathcal{F}$. It can be checked that it defines an equivalent relation on $\mathbb{E}_G(A,B)$, see \cite{JensenThomsen}. Then $KK_G(A,B)$ is the set of equivalence classes in $\mathbb{E}_G(A,B)$ under this equivalence relation.

\vspace{3mm}

Let $G$ be a compact Lie group and $T$ be its maximal torus. Let $i\colon T\rightarrow G$ be the inclusion from $T$ to $G$. Then every $G$-$\mathrm{C}^\ast$-algebra $A$ can be naturally considered as a $T$-$\mathrm{C}^\ast$-algebra via $i$, that is, $t.x=i(t)x$ where $t\in T$ and $x\in A$. Hence we have a map naturally induced from $i$, \[i^\ast\colon KK_G(A,B)\longrightarrow KK_T(A,B)\]for all $G$-$\mathrm{C}^\ast$-algebras $A$ and $B$. This map is also called the \emph{restriction map} and we will also make use of a more descriptive notation as follows: \[res^G_T\colon KK_G(A,B)\longrightarrow KK_T(A,B).\]

The goal of Sections \ref{section2.1} to \ref{section2.4} is to show that there is a left inverse $i_!\colon KK_T(A,B)\rightarrow KK_G(A,B)$ of $i^\ast\colon KK_G(A,B)\rightarrow KK_T(A,B)$. That is, 
\[i_!\circ i^\ast=1\colon KK_G(A,B)\rightarrow KK_G(A,B).\]Then we will prove our main Theorem \ref{main theorem2} in Section \ref{section2.5} which describes the subgroup $i^\ast(KK_G(A,B))$ by using \emph{divided difference operators}. The notion of divided difference operators will be introduced in Section \ref{section2.5}.

\subsection{Construction of $[i^\ast]\in KK_G(\mathbbm{C},C(G/T))$}\label{section2.1}

If $A$ is a $T$-$\mathrm{C}^\ast$-algebra, define $Ind^G_T(A)$ to be the $G$-$\mathrm{C}^\ast$-algebra of all continuous functions $f\colon G\rightarrow A$ such that $f(gt)=t^{-1}f(g)$ for all $g\in G$, $t\in T$ and $||f||$ vanishes at infinity. The $G$-action on $Ind^G_T(A)$ is by the left translation. Then there is a fairly natural way to define the \emph{induction} map \[ind^G_T\colon KK_T(A,B)\longrightarrow KK_G(Ind^G_T(A), Ind^G_T(B))\]for all $T$-$\mathrm{C}^\ast$-algebras $A$ and $B$. Its definition and properties will be explained in details in Section \ref{Wassermansection}. 

If $B$ is a $G$-$\mathrm{C}^\ast$-algebra, denote $Res^G_T(B)$ to be the $T$-$\mathrm{C}^\ast$-algebra by restricting the $G$-action to the $T$-action. It can be shown that for all $G$-$\mathrm{C}^\ast$-algebras $A$, $Ind^G_T(Res^G_T(A))$ is \emph{equivariantly isomorphic} to $A\otimes C(G/T)$, see Section \ref{Wassermansection}. 

We shall construct an element $[i^\ast]\in KK_G(\mathbb{C},C(G/T))$ corresponding to \[i^\ast\colon KK_G(A,B)\rightarrow KK_T(A,B).\]Define \[[i^\ast]=[C(G/T),id_\mathbb{C},0]\in KK_G(\mathbb{C},C(G/T))\]where $id_\mathbb{C}$ stands for the scalar multiplication and $C(G/T)$ is naturally viewed as a $G$-Hilbert $C(G/T)$-module. We need the following result by Wasserman \cite{Wasserman}:

\begin{theorem}[Wasserman] \label{Wasserman1}
Let $G$ be a compact group, and $T$ be its closed subgroup. If $A$ and $B$ are $G$-$\mathrm{C^\ast}$-algebras, then $KK_T(A,B)\cong KK_G(A,B\otimes C(G/T))$. Precisely speaking, if $x\in KK_T(A,B)$, then there is an isomorphism $x\mapsto j^\ast(ind^G_T(x))$ where $j^\ast$ is the map induced by the inclusion $j\colon A\cong A\otimes 1\longrightarrow A\otimes C(G/T)\cong Ind^G_T(A)$. And the inverse is given by $y\mapsto ev_\ast(res^G_T(y))$ for $y\in KK_G(A,B\otimes C(G/T))$ where $ev\colon B\otimes C(G/T)\rightarrow B$ is the evaluation at identity, i.e. $b\otimes f\mapsto bf(1)$. 
\end{theorem}

For a proof of it, see Section \ref{Wassermansection}. Let $\theta$ be the isomorphism $ev_\ast\circ res^G_T\colon KK_G(A,B\otimes C(G/T))\rightarrow KK_T(A,B)$.

\begin{lemma} \label{mainlemma}
For any element $x\in KK_G(A,B)$, \[\theta(x\otimes_\mathbb{C}[i^\ast])=i^\ast(x)\in KK_T(A,B).\]
\end{lemma}

\begin{proof}
It can be done by routine checking. Let $x=[E,\phi,F]\in KK_G(A,B)$, then \[x\otimes_\mathbb{C}[i^\ast]=[E\otimes C(G/T),\phi\otimes id,F\otimes id]\]where $E\otimes C(G/T)$ is the same as the external tensor product of two $G$-Hilbert modules and hence is a $G$-Hilbert $B\otimes C(G/T)$-module. \[\theta(x\otimes_\mathbb{C}[i^\ast])=ev_\ast\circ res^G_T(x\otimes_\mathbb{C}[i^\ast])=[(E\otimes C(G/T))\otimes_{ev}B,\phi\otimes id_\mathbb{C}\otimes id_B,F\otimes id\otimes id_B]\]where $(E\otimes C(G/T))\otimes_{ev} B$ is a $T$-Hilbert $B$-module. It is clear that $(E\otimes C(G/T))\otimes_{ev}B$ is isomorphic to $E$ as a $T$-Hilbert $B$-module. Let $f$ be the isomorphism from $(E\otimes C(G/T))\otimes_{ev}B$ to $E$. Then it is straightforward to check that \[f\circ (\phi\otimes id \otimes id_B)(a)=\phi(a)\circ f\]and \[f\circ (F\otimes id \otimes id_B)=F\circ f\]for any $a\in A$, $\phi$ is viewed as a $T$-equivariant map and $F$ is viewed as a $T$-Hilbert $B$-module map by restricting the $G$-action to the $T$-action. Hence, $\theta(x\otimes_\mathbb{C}[i^\ast])$ and $i^\ast(x)$ are unitarily equivalent in $\mathbb{E}_T(A,B)$ and our result follows.
\end{proof}

\subsection{Construction of $[i_!]\in KK_G(C(G/T),\mathbbm{C})$.}\label{section2.2}

$G/T$ is equipped with a $G$-equivariant complex structure corresponding to a choice of a positive root system relative to $(G/T)$, see Section 13 in \cite{Borel}. We can construct an \emph{equivariant Dolbeault element} $KK_G(C(G/T),\mathbb{C})$ as follows.

The $G$-action on $C(G/T)$ is defined by \[g.f(x)=f(g^{-1}x)\]for any $g\in G$, $x\in G/T$ and $f\in C(G/T)$. The $G$-action on any smooth $(0,\ast)$-form is defined by \[g.s(x)=g(s(g^{-1}x))\]where $g\in G$, $x\in G/T$ and $s$ is a smooth section of the vector bundle $\Omega^{(0,\ast)}$ of complex differential forms of degree $(0,\ast)$ over $G/T$. This action extends to an action on the Hilbert space $\mathbb{H}=L^2(G/T,\Omega^{(0,\ast)})$ by continuity. $\mathbb{H}$ is a Hilbert space graded by decomposing the forms into even and odd forms. Then let $D'=\partial+\bar{\partial}^\ast$ be the $G$-equivariant Dolbeault operator acting on smooth forms of $G/T$. It is an essentially self-adjoint operator of degree 1 (see \cite{HigsonRoe}). Note that it is an unbounded operator. Let $f$ be the real-valued function defined by $f(x)=x/ \sqrt{1+x^2}$. By functional calculus, define $F=f(D')$. $F$ is now a bounded operator acting on the smooth forms with compact supports. Extend such an action to $\mathbb{H}$ by continuity. By abuse of notation, this operator is denoted by $F$. Let $m$ be the function multiplication of $C_0(G/T)$ on $\mathbb{H}$. Then $[\mathbb{H},m,F]\in KK(C_0(G/T),\mathbbm{C})$. We call it the \emph{equivariant Dolbeault element} of $G/T$, denoted by $[\bar{\partial}_{G/T}]$. We define $[i_!]$ to be $[\bar{\partial}_{G/T}]$.

\subsection{The Kasparov Product $[i^\ast]\otimes_{C(G/T)}[i_!]\in KK_G(\mathbbm{C},\mathbbm{C})$}\label{section2.3}

Following the definition of Kasparov product, we can get the following:\[[i^\ast]\otimes_{C(G/T)}[i_!]=[C(G/T)\otimes_m L^2(G/T,\Omega^{(0,\ast)}),i,1\otimes F]\]where $C(G/T)\otimes_m L^2(G/T,\Omega^{(0,\ast)})$, as an internal tensor product of two Hilbert modules, is viewed as a $G$-Hilbert space. $G$ acts on it by\[g.(f\otimes_m h)=(g.f)\otimes_m (g.h)\]where $g\in G$, $f\in C(G/T)$ and $h\in C^\infty(G/T,\Omega^{(0,\ast)})$. We can extend this action to an action on $C(G/T)\otimes_m L^2(G/T,\Omega^{(0,\ast)})$ by continuity. $i$ is the scalar multiplication on $C(G/T)\otimes_m L^2(G/T,\Omega^{(0,\ast)})$.

In general, the Kasparov product is hard to compute. But in our particular case, Kasparov \cite{Kasparov2} showed the following result:

\begin{theorem} [Kasparov] \label{kasparovtheorem}
Let $G$ be a compact group and $M$ be a compact $G$-manifold. Let $[E]\in K_G^0(M)$ be an element in the equivariant $K$-theory of $M$ and let $[\bar{\partial}_M]\in KK_G(C(M),\mathbb{C})\cong K_0^G(M)$ be the equivariant Dolbeault element. Then \[[E]\otimes_{C(M)}[\bar{\partial}_M]=\mbox{G-index}((\bar{\partial}_M)_E)\]where $(\bar{\partial}_M)_E$ is the Dolbeault operator with coefficients in $E$.
\end{theorem}

Topologically, the element $[i^\ast]\in KK_G(\mathbb{C},C(G/T))\cong K_G^0(C(G/T))$ corresponds to the trivial $G$-bundle $E_0$ over $G/T$. The homogeneous pseudo-differential operator $D_{E_0}$ has $G$-index $1_G\in R(G)$ by a result of Bott, see \cite{Bott}. By Theorem \ref{kasparovtheorem}, we have the following result:

\begin{theorem}
$[i^\ast]\otimes_{C(G/T)}[i_!]=1\in KK_G(\mathbb{C},\mathbb{C}).$
\end{theorem}

\subsection{Push-pull operators}\label{section2.4}

Recall the notations from Section \ref{section2.1} that $\theta\colon KK_G(A,B\otimes C(G/T))\rightarrow KK_T(A,B)$ denotes the isomorphism by Wasserman's Theorem. Then let $\theta^{-1}\colon KK_T(A,B)\rightarrow KK_G(A,B\otimes C(G/T))$ be the inverse of $\theta$. Define $i_! \colon KK_T(A, B)\rightarrow KK_G(A,B)$ by \[i_!(y)=\theta^{-1}(y)\otimes_{C(G/T)} [i_!]\]for any $y\in KK_T(A,B)$. 

\begin{lemma}
$i_!\circ i^\ast=1$ as an action on $KK_G(A,B)$.
\end{lemma}

\begin{proof}
By Lemma \ref{mainlemma} and by associativity of Kasparov product,
\begin{eqnarray}
i_!(i^\ast(x))&=&i_!(\theta(x\otimes_{\mathbb{C}}[i^\ast]))\nonumber\\
  &=& (x\otimes_{\mathbb{C}}[i^\ast])\otimes_{C(G/T)}[i_!]\nonumber\\
  &=& x\otimes_{\mathbb{C}}([i^\ast]\otimes_{C(G/T)}[i_!])\nonumber\\
  &=& x\otimes_{\mathbb{C}}1 \nonumber\\
  &=& x \nonumber
\end{eqnarray}for all $x\in KK_G(A,B)$ as desired.
\end{proof}

Define $\sigma\colon KK_T(A,B)\longrightarrow KK_T(A,B)$ by \[\sigma=i^\ast \circ i_!.\]Some properties of $\sigma$ can be stated immediately.

\begin{lemma} \label{rho}
$\sigma^2=\sigma$ and $\sigma(i^\ast(x))=i^\ast(x)$ for any $x\in KK_G(A,B)$. 
\end{lemma}

\begin{proof}
By Section \ref{section2.3} and associativity of Kasparov product, \[([i_!]\otimes[i^\ast])\otimes([i_!]\otimes[i^\ast])=[i_!]\otimes([i^\ast]\otimes[i_!])\otimes[i^\ast]=
[i_!]\otimes[i^\ast].\]Now it is obvious that $\sigma^2=\sigma$ and $\sigma(i^\ast(x))=i^\ast(x)$ for any $x\in KK_G(A,B)$. 
\end{proof}

Let $R(T)$ be the character ring of the maximal torus $T$. The \emph{isobaric divided difference operators} $\delta_\alpha$ on $R(T)$ were introduced by Demazure \cite{Demazure}. The precise definitions were as follows. Let $\mathscr{R}$ be the root system of $(G,T)$. Let $s_\alpha\in W$ be the reflection element in the root $\alpha$. Let $\mathscr{X}(T)=\mathrm{Hom}(T,U(1))$ be the character group of $T$. We denote by $e^\lambda$ the element of $R(T)$ defined by a character $\lambda\in\mathscr{X}(T)$. The element $e^\lambda-e^{-\alpha}e^{s_\alpha(\lambda)}$ is divisible by $1-e^{-\alpha}$, then we can define a $\mathbb{Z}$-linear endomorphism $\delta_\alpha$ of $R(T)$ by \begin{eqnarray}
\delta_\alpha(u)=\frac{u-e^{-\alpha}s_\alpha(u)}{1-e^{-\alpha}} \label{firstequation}
\end{eqnarray}for all $u\in R(T)$. It has the following important property: \[\delta_\alpha^2=\delta_\alpha\]and \[\delta_\alpha(1)=1.\]Alternatively, in a series of earlier papers \cite{Demazure1}, \cite{Demazure2}, Demazure defined the operators \begin{eqnarray}\delta_\alpha '(u)=\frac{u-s_\alpha(u)}{1-e^{-\alpha}}.\end{eqnarray}It is easy to see that \[(\delta_\alpha ')^2=\delta_\alpha '\]and \[\delta_\alpha '(1)=0.\]For any $\omega\in W$ and any reduced expression $\omega=s_{\beta_1}s_{\beta_2}...s_{\beta_l}$ in terms of simple reflections, the composition $\delta_{\beta_1}\delta_{\beta_2}...\delta_{\beta_l}$ takes the same value $\partial_\omega$. Similarly, the composition $\delta_{\beta_1} ' \delta_{\beta_2} '...\delta_{\beta_l} '$ takes the same value $\partial_\omega '=e^{-\rho}\partial_\omega e^{-\rho}$, see \cite{Demazure}. For the longest element $\omega_0$, we have the \emph{top Demazure's operator} $\partial_{\omega_0}$.

$\partial_{\omega_0}$ is intimately related to the Weyl character Formula. We fix a basis of the root system and let \[\rho=\frac{1}{2}\sum_{\alpha\in\mathscr{R}^+}\alpha\]be the half-sum of all positive roots. The the Weyl character formula can be interpreted as the following formula: \begin{eqnarray} 
\mathrm{ch}(u)=\frac{\mathrm{A}(u)}{\mathrm{d}} \label{weylcharacter}\end{eqnarray}for all $u\in R(T)$. $\mathrm{A}(u)$ is the following alternating sums of elements in $R(T)$: \[\mathrm{A}(u)=\sum_{\omega\in W}(-1)^{l(\omega)}e^{-\rho}\omega(e^{\rho}u)\]where $l(w)$ is the length of the Weyl element $\omega$. $\mathrm{d}$ is defined as follows: \[\mathrm{d}=\prod_{\alpha\in\mathscr{R}^+}(1-e^{-\alpha}).\]In \cite{Demazure}, Demazure showed the following formula:\begin{eqnarray}\partial_{\omega_0}(u)=\frac{\mathrm{A}(u)}{\mathrm{d}}\label{Demazureformula}\end{eqnarray}for all $u\in R(T)$. 

\vspace{2mm}

If $A=\mathbb{C}$, $B=\mathbb{C}$, then $KK_T(\mathbb{C},\mathbb{C})\cong R(T)$ and $KK_G(\mathbb{C},\mathbb{C})\cong R(G)$. By our definition of $\sigma$ and the work of \cite{AtiyahBott} and \cite{Demazure}, $\sigma$ is exactly the same as $\partial_{\omega_0}$ acting on $R(T)$. In other words, the operator $\sigma\colon KK_T(A,B)\longrightarrow KK_T(A,B)$ can be interpreted as a generalization of both the Weyl character formula and the top Demazure's operator to Kasparov's $KK$-theory. 

\vspace{3mm}

We call a compact Lie group $G$ a \emph{Hodgkin group} if it is connected and has a torsion-free fundamental group. In \cite{Hodgkin}, Hodgkin proved the following result in equivariant $K$-theory: \[K_T^\ast(M)\cong R(T)\otimes_{R(G)}K_G^\ast(M)\] where $G$ is a Hodgkin group, $T$ is a maximal torus of $G$ and $M$ is any $G$-space which is locally contractible and of a finite covering dimension. Note that it is an isomorphism of $R(T)$-modules. The following generalization of the Hodgkin's result to $KK$-theory was due to A. Wasserman \cite{Wasserman}. See Section \ref{Rosenbergsection} for a proof of it.

\begin{theorem} [Wasserman]\label{Rosenberg_result}
Let $G$ be a Hodgkin group and $T$ be a maximal torus in $G$. For all $G$-$\mathrm{C}^\ast$-algebras $A$ and $B$, \[KK_T(A,B)\cong KK_G(A,B)\otimes_{R(G)}R(T).\]They are isomorphic as $R(T)$-modules. The map $KK_G(A,B)\otimes_{R(G)}R(T)\rightarrow KK_T(A,B)$ is given by $x\otimes a\mapsto a.i^\ast(x)$ where $i\colon T\rightarrow G$ is the inclusion map.
\end{theorem}

The next result is crucial for the constructions of divided difference operators in Section \ref{section2.5}.

\begin{theorem} \label{hecke algebra}
Assume that $G$ is a Hodgkin group. Identify the $R(T)$-modules $KK_T(A,B)$ and $KK_G(A,B)\otimes_{R(G)}R(T)$ via Theorem \ref{Rosenberg_result}, then $\sigma=1\otimes\partial_{\omega_0}$, where $1$ denotes the identity operator of $KK_G(A,B)$ and $\partial_{\omega_0}$ is the top Demazure's operator.
\end{theorem}

\begin{proof}
By the Wasserman's Isomorphism $\theta\colon KK_G(A,B\otimes C(G/T))\rightarrow KK_T(A,B)$ and Theorem \ref{Rosenberg_result}, we can identify $KK_G(A,B)\otimes_{R(G)}R(T)$ with $KK_G(A,B\otimes C(G/T))$. But $R(T)$ is isomorphic to $KK_G(\mathbb{C}, C(G/T))$. Hence we can consider $KK_G(A,B)\otimes_{R(G)}KK_G(\mathbb{C}, C(G/T))$ instead. Note that the relation $(xb)\otimes c=x\otimes (bc)\in KK_G(A,B)\otimes_{R(G)}KK_G(\mathbb{C}, C(G/T))$ where $x\in KK_G(A,B)$, $b\in R(G)$ and $c\in KK_G(\mathbb{C}, C(G/T))$ is equivalent to (after making identifications of $R(G)\cong KK_G(\mathbb{C},\mathbb{C})$) the associativity of Kasparov product $(x\otimes_\mathbb{C}b)\otimes_\mathbb{C}c=x\otimes_\mathbb{C}(b\otimes_\mathbb{C}c)$. Then this theorem is almost trivial. For any $x\otimes a\in KK_G(A,B)\otimes_{R(G)}R(T)$, the operator $1\otimes\partial_{\omega_0}$ acts on $KK_G(A,B)\otimes_{R(G)}KK_G(\mathbb{C},C(G/T))$ by \begin{eqnarray}
1\otimes\partial_{\omega_0}(x\otimes a)&=&x\otimes\partial_{\omega_0}a\nonumber\\
  &=&x\otimes(a\otimes_{C(G/T)}[i_!]\otimes_\mathbb{C}[i^\ast]).\nonumber
\end{eqnarray}  In terms of Kasparov product, $x\otimes_{\mathbb{C}}(a\otimes_{C(G/T)}[i_!]\otimes_\mathbb{C}[i^\ast])=(x\otimes_\mathbb{C}a)\otimes_{C(G/T)}[i_!]\otimes_\mathbb{C}[i^\ast]$. But then $(x\otimes_\mathbb{C}a)\otimes_{C(G/T)}[i_!]\otimes_\mathbb{C}[i^\ast]$ is essentially the same as $\sigma(a.i^\ast(x))$.
\end{proof}

The next result is analogous to a result by Snaith \cite{Snaith}.

\begin{lemma} \label{Lemma11}
Let $\tilde{T}$ be a torus and $s\colon\tilde{T}\rightarrow T$ a covering homomorphism. Then the map $s^\ast\colon KK_T(A,B)\rightarrow KK_{\tilde{T}}(A,B)$ is injective for all $T$-$\mathrm{C}^\ast$-algebras $A$ and $B$.
\end{lemma}

\begin{proof}
Let $t\colon C\rightarrow \tilde{T}$ be the kernel of $s$. Let $\mathbb{E}_T$ be \[\mathbb{E}_T=\prod_{\lambda\in\mathscr{X}(C)}\mathbb{E}_T(A,B)\]where $\mathscr{X}(C)$ is the character group of $C$. We write an object of $\mathbb{E}_T$ as an $\mathscr{X}(C)$-tuple $([E_\lambda,\phi_\lambda,F_\lambda])_{\lambda\in\mathscr{X}(C)}$, where each $[E_\lambda,\phi_\lambda,F_\lambda]$ is an element in $\mathbb{E}_T(A,B)$. The restriction homomorphism $t^\ast\colon\mathscr{X}(\tilde{T})\rightarrow\mathscr{X}(C)$ is surjective, see \cite{Snaith}. We choose a set-theoretic left inverse $\tau$. Let $V_\zeta$ be the one-dimensional $\tilde{T}$-module defined by a character $\zeta\in\mathscr{X}(\tilde{T})$. Let $\mathbb{E}_{\tilde{T}}=\mathbb{E}_{\tilde{T}}(A,B)$ and $[E,\phi,F]\in\mathbb{E}_{\tilde{T}}$. Since $C$ acts trivially on the $T$-$\mathrm{C}^\ast$-algebra $B$, the $C$-invariant subspace $E^C$ of $E$ is a well-defined $T$-Hilbert $B$-module. For all objects $[E,\phi,F]$ in $\mathbb{E}_{\tilde{T}}$, define $\nu\colon\mathbb{E}_{\tilde{T}}\rightarrow\mathbb{E}_T$ by \[\nu([E,\phi,F])=[Hom(V_{\tau(\lambda)},E)^C,\tilde{\phi}_\lambda,\tilde{F}_\lambda]_{\lambda\in\mathscr{X}(C)}\]where $Hom(V_{\tau(\lambda)},E)$ is the set of all $\tilde{T}$-maps from $V_{\tau(\lambda)}$ to $E$. It is a $\tilde{T}$-Hilbert $B$-module with the $B$-module structure defined by \[fb(v)=f(v)b\]for all $b\in B$ and $v\in V_{\tau(\lambda)}$. $Hom(V_{\tau(\lambda)},E)^C$ is the $C$-invariant subspace of $Hom(V_{\tau(\lambda)},E)$. Since $C$ acts trivially on the $T$-$\mathrm{C}^\ast$-algebra $B$, $Hom(V_{\tau(\lambda)},E)^C$ is a $T$-Hilbert $B$-module. $\tilde{\phi}_\lambda\colon A^C\rightarrow\mathrm{B}(Hom(V_{\tau(\lambda)},E)^C)$ where $A^C$ is a $T$-$\mathrm{C}^\ast$-algebra by taking $C$-invariants of the $\tilde{T}$-action on $A$, is defined by \[(\tilde{\phi}_\lambda (a)f)(v)=\phi(a)(f(v))\]for all $f\in Hom(V_{\tau(\lambda)},E)$, $v\in V_{\tau(\lambda)}$ and $\lambda\in\mathscr{C}(C)$. It is easy to check that $\tilde{\phi}_\lambda$ is a $T$-$\ast$-homomorphism. Similarly, $\tilde{F}_\lambda \in\mathrm{B}(Hom(V_{\tau(\lambda)},E)^C)$ is defined by \[(\tilde{F}_\lambda(f))(v)=F(f(v))\]for all $f\in Hom(V_{\tau(\lambda)},E)^C$ and $v\in V_{\tau(\lambda)}$. Again, it is routine to check that $\tilde{F}_\lambda$ is a $T$-Hilbert $B$-module map.

For all objects $[E_\lambda,\phi_\lambda,F_\lambda]_{\lambda\in\mathscr{X}(C)}$ in $\mathbb{E}_T$, define $\mu\colon\mathbb{E}_T\rightarrow\mathbb{E}_{\tilde{T}}$ by \[\mu([E_\lambda,\phi_\lambda,F_\lambda]_{\lambda\in\mathscr{X}(C)})=\bigoplus_{\lambda\in\mathscr{X}(C)}[V_{\tau(\lambda)}\otimes s^\ast E_\lambda,id\otimes s^\ast\phi_\lambda, id\otimes s^\ast F_\lambda]\]where $s^\ast E_\lambda$ is regarded as a $\tilde{T}$-Hilbert $B$-module through $s$. Likewise, $s^\ast \phi_\lambda$ and $s^\ast F_\lambda$ are regarded as $\tilde{T}$-$\ast$-homomorphisms and $\tilde{T}$-Hilbert $B$-module maps via $s$ respectively. $V_{\tau(\lambda)}\otimes s^\ast E_\lambda$ is the external tensor product of $V_{\tau(\lambda)}$ (as a $\tilde{T}$-Hilbert space) and $s^\ast E_\lambda$. Hence it is a $\tilde{T}$-Hilbert $B$-module itself after identifying $\mathbb{C}\otimes B$ with $B$ as $\tilde{T}$-$\mathrm{C}^\ast$-algebras. 

Then, for all $[E_\lambda,\phi_\lambda,F_\lambda]_{\lambda\in\mathscr{X}(C)}$ in $\mathbb{E}_T$, \[\nu(\mu([E_\lambda,\phi_\lambda,F_\lambda]_\lambda))=[Hom(V_{\tau(\psi)},\bigoplus_\lambda V_{\tau(\lambda)}\otimes s^\ast E_\lambda)^C, \bigoplus_\lambda(\widetilde{id\otimes s^\ast\phi_\lambda})_\psi, \bigoplus_\lambda(\widetilde{id \otimes s^\ast F_\lambda})_\psi]_{\psi\in\mathscr{X}(C)}.\]And 
\begin{eqnarray}
Hom(V_{\tau(\psi)},\bigoplus_\lambda V_{\tau(\lambda)}\otimes s^\ast E_\lambda)^C
  &=& \bigoplus_\lambda Hom(V_{\tau(\psi)},V_{\tau(\lambda)}\otimes s^\ast E_\lambda)^C\nonumber\\
  &=& \bigoplus_\lambda Hom(V_{\tau(\psi)},V_{\tau(\lambda)})^C \otimes (s^\ast E_\lambda)^C \nonumber\\
  &=& E_\psi. \nonumber
\end{eqnarray}From here it is easily verified that \begin{eqnarray}
(\widetilde{id\otimes s^\ast \phi_\lambda})_\psi=\phi_\psi\nonumber\\
(\widetilde{id\otimes s^\ast F_\lambda})_\psi= F_\psi\nonumber
\end{eqnarray}if $\lambda=\psi$. And $(\widetilde{id\otimes s^\ast \phi_\lambda})_\psi=0$, $(\widetilde{id\otimes s^\ast F_\lambda})_\psi=0$ otherwise. and Hence, \[\nu\mu([E_\lambda,\phi_\lambda,F_\lambda]_\lambda)=[E_\lambda,\phi_\lambda,F_\lambda]_\lambda.\]

For all objects $[E,\phi,F]$ in $\mathbb{E}_{\tilde{T}}$, \[\mu(\nu([E,\phi,F]))=\bigoplus_\lambda[V_{\tau(\lambda)}\otimes s^\ast(Hom(V_{\tau(\lambda)},E)^C), id\otimes s^\ast\tilde{\phi}_\lambda, id\otimes s^\ast\tilde{F}_\lambda].\]We have \[\bigoplus_\lambda V_{\tau(\lambda)}\otimes s^\ast(Hom(V_{\tau(\lambda)},E)^C)\cong E\] by virtue of Chapter III (6.4) in \cite{brocker}. From here it is easily verified that \begin{eqnarray}
\bigoplus_\lambda id \otimes s^\ast\tilde{\phi}_\lambda\cong \phi\nonumber\\
\bigoplus_\lambda id \otimes s^\ast\tilde{F}_\lambda\cong F.\nonumber
\end{eqnarray}Hence, we have \[\mu\nu([E,\phi,F])=[E,\phi,F].\]We conclude that the  categories $\mathbb{E}_{\tilde{T}}$ and $\mathbb{E}_{T}$ are equivalent. 

If two elements in $x,y\in\mathbb{E}_{\tilde{T}}(A,B)$ are homotopic, i.e. they represent the same class in $KK_{\tilde{T}}(A,B)$, then there exists an element $a\in\mathbb{E}_{\tilde{T}}(A,B[0,1])$ such that $(ev_0)_\ast (a)=x$ and $(ev_1)_\ast (a)=y$, where $ev_j \colon B([0,1])\rightarrow B$ is the evaluation at $j$, $j=0,1$. We consider the element $\nu(a)=(a_\lambda)_{\lambda\in\mathscr{X}(C)}\in\prod_\lambda \mathbb{E}_T(A,B([0,1]))$. Then $(ev_0)_\ast ((a_\lambda)_{\lambda\in\mathscr{X}(C)})$ and $(ev_1)_\ast ((a_\lambda)_{\lambda\in\mathscr{X}(C)})$ are homotopic in $\prod_\lambda \mathbb{E}_T(A,B)$. A couple of definition-tracing arguments show that $\mu((ev_0)_\ast ((a_\lambda)_\lambda))=x$ and $\mu((ev_1)_\ast ((a_\lambda)_\lambda))=y$ in $\mathbb{E}_{\tilde{T}}(A,B)$. It means that there is a well-defined injective map from $KK_{\tilde{T}}(A,B)$ to $\oplus_\lambda KK_T(A,B)$. A very similar argument starting from two homotopic elements in $\prod_\lambda \mathbb{E}_T(A,B)$ shows the reverse inclusion and hence we obtain \[\bigoplus_{\lambda\in\mathscr{X}(C)} KK_T(A,B)\cong KK_{\tilde{T}}(A,B).\]The isomorphism $\oplus_\lambda KK_T(A,B)\rightarrow KK_{\tilde{T}}(A,B)$ is defined by \[[E_\lambda,\phi_\lambda,F_\lambda]_{\lambda\in\mathscr{X}(C)}\mapsto \sum_{\lambda\in\mathscr{X}(C)}[V_{\tau(\lambda)}]\otimes_\mathbb{C} s^\ast([E_\lambda,\phi_\lambda,F_\lambda])\]where $[V_{\tau(\lambda)}]\in R(\tilde{T})\cong KK_{\tilde{T}}(\mathbb{C},\mathbb{C})$ and $\otimes_\mathbb{C}$ is the Kasparov product over $\mathbb{C}$. In particular, setting $A=\mathbb{C}$ and $B=\mathbb{C}$ gives \[\bigoplus_{\lambda\in\mathscr{X}(C)}R(T)\cong R(\tilde{T})\] and hence \[\bigoplus_{\lambda\in\mathscr{C}(C)}KK_T(A,B)\cong R(\tilde{T})\otimes_{R(T)}KK_T(A,B).\]Hence, we have \[KK_{\tilde{T}}(A,B)\cong R(\tilde{T})\otimes_{R(T)}KK_T(A,B)\]which proves the lemma.
\end{proof}

\subsection{Main Theorem}\label{section2.5}

In this section, we shall show our main theorems, Theorem \ref{main theorem} and Theorem \ref{main theorem2}.

Let $\alpha$ be a root, $G_\alpha$ be the centralizer in $G$ of $\ker\alpha$ and $i_\alpha \colon T\rightarrow G_\alpha$ be the inclusion. Motivated by the definition of $i_!$, we want to define a `pushforward' map $i_{\alpha,!}\colon KK_T(A,B)\rightarrow KK_{G_\alpha}(A,B)$ for every root $\alpha$. First, we choose a complex structure on $G_\alpha /T$. We do this by identifying $G_\alpha /T$ with the complex homogeneous space $(G_\alpha)_{\mathbb{C}}/B$ where $B_\alpha$ is the Borel subgroup of $(G_\alpha)_{\mathbb{C}}$ generated by $T_{\mathbb{C}}$ and the root space $\mathfrak{g}_{\mathbb{C}}^{-\alpha}$. Then $[i_{\alpha,!}]$ is defined in the same way as $[i_!]$ in Section \ref{section2.2}. Moreover, the map $i_{\alpha,!}\colon KK_T(A,B)\rightarrow KK_{G_\alpha}(A,B)$ is also defined in the same way as $i_!$, see Section \ref{section2.4}. 

Define $\sigma_\alpha \colon KK_T(A,B)\longrightarrow KK_T(A,B)$ by \[\sigma_\alpha=i_\alpha^\ast\circ i_{\alpha,!}\]for every root $\alpha$.

By Lemma \ref{rho} for $G=G_\alpha$, $\sigma_\alpha$ has the properties that $\sigma_\alpha^2=\sigma_\alpha$ and $\sigma_\alpha(i_\alpha^\ast(x))=i_\alpha^\ast(x)$ for $x\in KK_{G_\alpha}(A,B)$. 

\begin{definition}
$\sigma_\alpha$ as defined above is called the \emph{divided difference operator corresponding to the root $\alpha$}. The set $\{\sigma_\alpha| \alpha\in\mathscr{R}\}$ is called the set of divided difference operators which act on $KK_T(A,B)$.
\end{definition}

Under the same assumptions as in Theorem \ref{hecke algebra} we have $\sigma_\alpha=1\otimes \delta_\alpha$ for all roots $\alpha$.

\begin{remark}
The power of equivariant $KK$-theory comes from the fact that it generalizes both equivariant $K$-theory and equivariant $K$-homology. In $K$-theory, when $A=\mathbb{C}$ and $B=C(M)$ where $M$ is a compact $G$-space, our set of divided difference operators specializes to a set of divided difference operators in $T$-equivariant $K$-theory of $M$, $K_T(M)$, which was first defined in \cite{Sjamaar}. On the other hand, if $B=\mathbb{C}$, then it simply means that we have now abstractly defined a set of divided difference operators in $K_T^0(A)$.
\end{remark}

\vspace{3mm} 

Let $\mathscr{E}=End_{R(G)}(R(T))$ be the $R(G)$-algebra of $R(G)$-linear endomorphisms of $R(T)$. Let $\mathscr{D}$ be the subalgebra of $\mathscr{E}$ generated by the isobaric divided difference operators $\delta_\alpha$ and the elements of $R(T)$ (as multiplication operators). By definitions of $\partial_\omega$, $\partial_\omega '$ in Section \ref{section2.4}, we have $\partial_\omega, \partial_\omega '\in\mathscr{D}$ for all $\omega$. As a ring $\mathscr{D}$ is isomorphic to the Hecke algebra over $\mathbb{Z}$ of the extended affine Weyl group $\mathscr{X}(T)\rtimes W$, see \cite{Kazhdan}. In \cite{Sjamaar} $\mathscr{D}$ is called the Hecke algebra. 

The \emph{augmentation left ideal} of $\mathscr{D}$ is the annihilator of the identity element $1\in R(T)$, that is \[I(\mathscr{D})=\{\Delta\in\mathscr{D}|\Delta(1)=0\}.\]By (\ref{firstequation}), $\mathscr{D}$ contains the group ring $\mathbb{Z}[W]$ when $\mathbb{Z}[W]$ is viewed as an algebra of endomorphisms of $R(T)$. Hence $I(\mathscr{D})$ naturally contains the augmentation ideal $I(W)$ of $\mathbb{Z}[W]$. Since $\partial_\omega '(1)=0$ for $\omega\neq 1$, $I(\mathscr{D})$ contains all $\partial_\omega '$ when $\omega\neq 1$.

Some properties of $\mathscr{D}$ and $I(\mathscr{D})$ are noted as follows.

\begin{theorem}[Harada, Landweber and Sjamaar] \label{sjamaartheorem}

(i) $(\partial_\omega)_{\omega\in W}$ is a basis of the left $R(T)$-module $\mathscr{D}$.

(ii) $(\partial_\omega ')_{\omega\in W}$ is a basis of the left $R(T)$-module $\mathscr{D}$.

(iii) $(\partial_\omega)_{\omega\neq 1}$ is a basis of the left $R(T)$-module $I(\mathscr{D})$.
\end{theorem}

Let $N$ be a left $\mathscr{D}$-module. We say an element of $N$ is \emph{$\mathscr{D}$-invariant} if it is annihilated by all operators in the augmentation left ideal $I(\mathscr{D})$. Let $N^{I(\mathscr{D})}$ be the group of invariants. By Theorem \ref{sjamaartheorem}, \[N^{I(\mathscr{D})}=\{n\in N| \partial_\omega '(n)=0,\mbox{for all} \mbox{ }\omega\neq 1\}.\]Since $I(\mathscr{D})$ contains the augmentation left ideal $I(W)$ of $\mathbb{Z}[W]$, we have \begin{eqnarray}N^{I(\mathscr{D})}\subseteq N^W\label{inequality}\end{eqnarray}where $N^W$ contains elements that are invariant under the Weyl group action.

We now show that $KK_T(A,B)$ is equipped with a left $\mathscr{D}$-module structure in Theorem \ref{main theorem}. Then, by (\ref{inequality}), we have the following \begin{eqnarray}KK_T(A,B)^{I(\mathscr{D})}\subseteq KK_T(A,B)^W. \label{inequality2} \end{eqnarray}We will discuss (\ref{inequality2}) in Section \ref{section2.6}.

\begin{theorem} \label{main theorem}
The operators $\sigma_\alpha$ for $\alpha\in\mathscr{R}$, together with the natural $R(T)$-module structure generate an unique $\mathscr{D}$-module structure on $KK_T(A,B)$.
\end{theorem} 

\begin{proof}
The proof is very similar to Prop. 4.5 in \cite{Sjamaar} and is essentially an application of Theorem \ref{Rosenberg_result}, Theorem \ref{hecke algebra} and Lemma \ref{Lemma11}. First, assume that $G$ is a Hodgkin group. Idenitfy $KK_T(A,B)$ with $KK_G(A,B)\otimes_{R(G)} R(T)$ through the isomorphism of Theorem \ref{Rosenberg_result}. Let \[\mathscr{E}(A,B)=KK_G(A,B)\otimes\mathscr{E}.\]Then the map $\mathscr{D}\rightarrow\mathscr{E}(A,B)$ defined by $\Delta\mapsto 1\otimes\Delta$, where $1$ is the identity map of $KK_G(A,B)$, is a well-defined algebra homomorphism. Since $\sigma_\alpha=1\otimes\delta_\alpha$, $\sigma_\alpha$ generates a well-defined action of $\mathscr{D}$ on $KK_T(A,B)$.

If $G$ is not a Hodgkin group, we choose a covering $s\colon\tilde{G}\rightarrow G$ such that $\tilde{G}$ is a Hodgkin group. By Lemma \ref{Lemma11} the pullpack \[s^\ast\colon KK_T(A,B)\rightarrow KK_{\tilde{T}}(A,B)\]is injective, where $\tilde{T}$ is the maximal torus $s^{-1}(T)$ of $\tilde{G}$. Let $\tilde{\sigma}_\alpha=\tilde{i}_\alpha^\ast \circ \tilde{i}_{\alpha,!}$ be the operator on $KK_{\tilde{T}}(A,B)$ corresponding to $\alpha$, where $\tilde{i}_\alpha\colon\tilde{T}\rightarrow\tilde{G}_\alpha$ is the inclusion. By the naturality properties of $i_\alpha^\ast$ and $i_{\alpha,!}$\begin{eqnarray}s^\ast \sigma_\alpha&=&\tilde{\sigma}_\alpha s^\ast. \label{natural}\end{eqnarray}By Lemma 2.4 \cite{Sjamaar}, $s$ induces an injective algebra homomorphism \[\overline{s}\colon\mathscr{D}\rightarrow\mathscr{\tilde{D}}.\]We already know that $\tilde{\sigma}_\alpha$ generate a well-defined $\tilde{\mathscr{D}}$-action on $KK_{\tilde{T}}(A,B)$. This $\tilde{\mathscr{D}}$-module structure on $KK_{\tilde{T}}(A,B)$ is unique due to Theorem \ref{hecke algebra}. The restriction of the $\tilde{\mathscr{D}}$-action to the subalgebra $\mathscr{D}$ preserves the submodule $KK_T(A,B)$ and by (\ref{natural}), the elements $\sigma_\alpha$ act in the required fashion. It is clear that the $\mathscr{D}$-module structure on $KK_T(A,B)$ so defined is unique.
\end{proof}

By Theorem \ref{main theorem}, it is now clear that if $A=B=\mathbb{C}$, our set of divided difference operators $\sigma_\alpha$ that acts on $KK_T(A,B)=KK_T(\mathbb{C},\mathbb{C})\cong R(T)$ is the same as the set of Demazure's operators $\delta_\alpha$.

If $G$ is a Hodgkin group, let $\mathscr{U}=\mathscr{D}$-Mod and $\mathscr{B}=R(G)$-Mod be the categories of left modules over the rings $\mathscr{D}$ and $R(G)$ respectively. Before stating our next theorem, we invoke the following result shown in \cite{Sjamaar}.

\begin{theorem}[Harada, Landweber and Sjamaar] \label{Sjamaar_result}
If $G$ is a Hodgkin group, then the functor $\mathscr{G}\colon\mathscr{B}\rightarrow\mathscr{U}$ defined by \[B\mapsto B\otimes_{R(G)}R(T)\] is an equivalence with inverse $\mathscr{F}\colon\mathscr{U}\rightarrow\mathscr{B}$ given by \[A\mapsto Hom_{\mathscr{D}}(R(T),A).\]Moreover, $\mathscr{F}$ is naturally isomorphic to the functor $\mathscr{J}\colon\mathscr{U}\rightarrow\mathscr{B}$ given by \[A\mapsto A^{I(\mathscr{D})}.\]
\end{theorem}

The following result describes $KK_G(A,B)$ as a direct summand of $KK_T(A,B)$. More precisely, $KK_G(A,B)$ is isomorphic to $KK_T(A,B)$ annihilated by the set of divided difference operators.

\begin{theorem} \label{main theorem2}
For all $G$-$\mathrm{C}^\ast$-algebras $A$ and $B$, the map $i^\ast$ is an isomorphism from $KK_G(A,B)$ onto $KK_T(A,B)^{I(\mathscr{D})}$ where $i$ is the inclusion $T\rightarrow G$.
\end{theorem}

\begin{proof}
First assume that $G$ is a Hodgkin group, consider the $\mathscr{D}$-module $A=KK_T(A,B)$ and the $R(G)$-module $B=KK_G(A,B)$. By Theorem \ref{Rosenberg_result}, \[\mathscr{G}(B)=A.\]Hence, by Theorem \ref{Sjamaar_result}, \[B\cong\mathscr{F}(A)\cong\mathscr{J}(A)=A^{I(\mathscr{D})}.\]If $G$ is not a Hodgkin group, we use the same trick as in the proof of Theorem \ref{main theorem} to get our desired result.
\end{proof}

As a simple application of Theorem \ref{main theorem2}, in the case when $A$ is a $G$-$\mathrm{C}^\ast$-algebra and $B=\mathbb{C}$, we have the following corollary.

\begin{corollary} If $A$ is a $G$-$\mathrm{C}^\ast$-algebra, then \[K_G^0(A)\cong K_T^0(A)^{I(\mathscr{D})}.\]\end{corollary}

In particular, if $A=C(M)$ where $M$ is a compact $G$-manifold, then we have

\begin{corollary}Let $M$ be a compact $G$-manifold, then
\[K_0^G(M)\cong K_0^T(M)^{I(\mathscr{D})}.\]
\end{corollary}

\subsection{The difference between $KK_T(A,B)^{I(\mathscr{D})}$ and $KK_T(A,B)^W$}\label{section2.6}

Note that if $A=B=\mathbb{C}$, Theorem \ref{main theorem2} gives the following result: \[R(G)\cong R(T)^{I(\mathscr{D})}.\]But the isomorphism $R(G)\cong R(T)^W$ implies that in the case of the character ring of $T$, $R(T)^W=R(T)^{I(\mathscr{D})}$. One may wonder whether this result generalizes to the equivariant $KK$-group for any $G$-$\mathrm{C}^\ast$-algebras $A$ and $B$. But an example given by McLeod \cite{McLeod} showed that it is far from being true for equivariant $K$-theory, let alone equivariant $KK$-theory. See also \cite{Sjamaar} for a generalization of a McLeod's example. In this subsection, we shall see that, under some restrictive conditions, $KK_T(A,B)^{I(\mathscr{D})}$ is isomorphic to $KK_T(A,B)^W$.

\vspace{3mm}

A \emph{compact Hamiltonian $G$-space} is a compact symplectic manifold $(M,\omega)$ on which $G$ acts by symplectomorphisms, together with a $G$-equivariant moment map $\phi:M\rightarrow \mathfrak{g}^\ast$ satisfying the Hamilton's equation: \[\langle d\phi,X\rangle=\iota_{X'} \omega, \mbox{  }\forall X\in\mathfrak{g}\]where $G$ acts on $\mathfrak{g}^\ast$ by the coadjoint action and $X'$ denotes the vector field on $M$ generated by $X\in\mathfrak{g}$.

If $M$ is a compact Hamiltonian $G$-manifold, then the restriction map $K_T(M)\rightarrow K_T(M^T)$ induced by $M^T\rightarrow M$ is injective by Theorem 2.5 in \cite{HHaradaLandweber}. Based on this result, it was shown in \cite {Sjamaar} that \begin{eqnarray}K_G(M)\cong K_T(M)^W.\label{symplecticcase}\end{eqnarray}

In \cite{Kasparov2}, Kasparov constructed a map $\tau\colon KK_G(C(M),\mathbb{C})\longrightarrow KK_G(\mathbb{C},C(M))$ for any even-dimensional compact $G$-manifold $M$ equipped with a $G$-equivariant $\mbox{spin}^c$-structure and used it to show that it is an isomorphism in $G$-equivariant $KK$-theory: \begin{eqnarray} KK_G(C(M),\mathbb{C})\cong KK_G(\mathbb{C},C(M)).\label{poincareduality}\end{eqnarray}It is called \emph{Poincar\'{e} duality} in equivariant $KK$-theory.

For a compact Hamiltonian $G$-manifold $M$ equipped with a $G$-equivariant symplectic form $\omega$, there is a $G$-equivariant almost complex structure naturally associated with $\omega$. It is canonical in the sense that it is unique up to homotopy. We obtain a $G$-equivariant $\mbox{spin}^c$-structure on $M$ by this equivariant almost complex structure. Thus, we can combine Kasparov's result (\ref{poincareduality}) with (\ref{symplecticcase}) to give the following corollary.

\begin{corollary}
If $M$ is a compact Hamiltonian $G$-manifold, then \[K^G_0(M)\cong K^T_0(M)^W\]where $K^G_0(M)$ is the $G$-equivariant $K$-homology of $M$.
\end{corollary}

Finally, we state some criteria for $KK_G(A,B)$ to be isomorphic to $KK_T(A,B)^W$. Recall that $\mbox{d}=\prod_{\alpha\in\mathscr{R}^+}(1-e^{-\alpha})\in R(T)$ is the Weyl denominator in (\ref{weylcharacter}).

\begin{lemma} \label{cheaplemma}
Assume that the Weyl denominator $\mbox{d}=\prod_{\alpha\in\mathscr{R}^+}(1-e^{-\alpha})\in R(T)$ is not a zero divisor in the $R(T)$-module $KK_T(A,B)$, then the map $i^\ast$ is an isomorphism from $KK_G(A,B)$ to $KK_T(A,B)^W$ where $i$ is the inclusion $T\rightarrow G$.
\end{lemma}

\begin{proof}
It follows immediately from Lemma 3.5 in \cite{Sjamaar}.
\end{proof}

The following corollary is immediate by Lemma \ref{cheaplemma}. It is a generalization of Theorem 4.4 in McLeod's paper \cite{McLeod}.

\begin{corollary}
If $KK_T(A,B)$ is a free module over $R(T)$, then \[KK_G(A,B)\cong KK_T(A,B)^W.\]
\end{corollary}

\section{Proof of Theorem \ref{Wasserman1}}\label{Wassermansection}

Theorem \ref{Wasserman1} is a version of Frobenius Reciprocity in equivariant $KK$-theory. As promised in Section \ref{section1} a proof will be provided here. We will only prove it for the case that $G$ is a compact group and $A$, $B$ are $G$-$\mathrm{C^\ast}$-algebras. 

Recall from Section \ref{section2} that if $A,B$ are $G$-$\mathrm{C}^\ast$-algebras, the we have the restriction map: \[res^G_T\colon KK_G(A,B)\rightarrow KK_T(A,B)\]which is defined by sending $x=[E,\phi,F]\in KK_G(A,B)$ to $x|_T=[E|_T,\phi |_T,F|_T]\in KK_T(A,B)$ where $E|_T$ is regarded as a $T$-Hilbert $B$-module. $\phi$ is regarded as a $T$-$\ast$ homomorphism and $F$ is regarded as a $T$-bounded operator in $\mathrm{B}(E|_T)$. To avoid notational confusion, we will also use the notations $Res^G_T E$, $Res^G_T F$, $Res^G_T \phi$ for $E|_T$, $F|_T$, $\phi |_T$ respectively.

On the other hand, if $M$ is a $T$-$\mathrm{C}^\ast$-algebra, then $Ind^G_T(M)$ is the $G$-$\mathrm{C}^\ast$-algebra of all continuous functions $f\colon G\rightarrow M$ such that $f(gh)=h^{-1}f(g),\forall g\in G,h\in T$ and such that $\parallel f\parallel$ vanishes at infinity. Since we are dealing with the case that $G/T$ is compact, the $\mathrm{C}^\ast$-norm of each element in $Ind^G_T(M)$ is just the maximum norm. The $G$-action on $Ind^G_T(M)$ is the left translation.

If $A$ is an $G$-$\mathrm{C}^\ast$-algebra, then $Ind^G_T(Res^G_T(A))$ is \emph{equivariantly isomorphic} to $A\otimes C(G/T)$. We denote the isomorphism from $Ind^G_T(Res^G_T(A))$ to $A\otimes C(G/T)$ by $\Phi$. More explicitly, if $F_A\in Ind^G_T(Res^G_T(A))$, then $\Phi(F_A)([g])=gF_A(g)$. The inverse map $\Phi^{-1}\colon A\otimes C(G/T)\rightarrow Ind^G_T(Res^G_T (A))$ is defined as follows: for $a\otimes f\in A\otimes C(G/T)$, $\Phi^{-1}(a\otimes f)(g)=f(g)g^{-1}a$. 

We are going to describe an \emph{induction} map from the $T$-equivariant $KK$-theory to the $G$-equivariant $KK$-theory  for any $G$-$\mathrm{C}^\ast$-algebras $A,B$.

Let $E$ is an $T$-Hilbert $B$-module, define $\tilde{E}:= Ind^G_T E$ by \[Ind^G_T E=\{f_E\colon G\rightarrow E\mid f(gt)=t^{-1}f(g)\}.\]It has an $Ind^G_T B$-valued inner product defined by \[\langle f_E,f'_E\rangle(g):=\langle f_E(g),f'_E(g)\rangle\] for any $f_E,f'_E\in Ind^G_T(E)$ and $g\in G$.

\begin{lemma}
$\tilde{E}$ is an $G$-Hilbert $Ind^G_T B$-module. 
\end{lemma}

\begin{proof}
 For $f_B\in Ind^G_T(B)$ and $f_E\in Ind^G_T(E)$, we have 
\begin{eqnarray}
(f_Ef_B)(gt)&=&f_E(gt)f_B(gt)                   \nonumber\\
            &=&(t^{-1}f_E(g))(t^{-1}f_B(g))     \nonumber\\
            &=&t^{-1}(f_E(g)f_B(g))             \nonumber\\
            &=&t^{-1}(f_E f_B)(g).               \nonumber   
\end{eqnarray}Hence $f_E f_B\in Ind^G_T(E)$. Moreover,
\begin{eqnarray}
\langle f_E,f'_E\rangle(gt)&=&\langle f_E(gt),f'_E(gt)\rangle              \nonumber\\
              &=&\langle t^{-1}f_E(g),t^{-1}f'_E(g)\rangle    \nonumber\\
              &=&t^{-1}\langle f_E(g),f'_E(g)\rangle          \nonumber\\
              &=&t^{-1}(\langle f_E,f'_E\rangle (g)).           \nonumber
\end{eqnarray}Hence, $\langle f_E,f'_E\rangle \in Ind^G_T(B)$. It is easy to check that $\langle f_E,f'_Ef_B\rangle=\langle f_E,f'_E\rangle f_B$ and other properties of the Hilbert $Ind^G_T B$-module are easily verified. The $G$-action on $Ind^G_T(E)$ is the left translation for all $f_E\in Ind^G_T(E)$. Then \begin{eqnarray}
g\langle f_E,f'_E\rangle(x)&=&\langle f_E,f'_E\rangle (g^{-1}x)         \nonumber\\
              &=&\langle f_E(g^{-1}x),f'_E(g^{-1}x)\rangle\nonumber\\
              &=&\langle gf_E(x),gf'_E(x)\rangle.          \nonumber
\end{eqnarray}Similarly, other properties of the $G$-Hilbert module structure are easily verified.\end{proof}

If $\phi\colon A\rightarrow\mathrm{B}(E)$ a $T$-$\ast$-homomorphism, define  $\tilde{\phi}:= Ind^G_T \phi\colon Ind^G_T A\rightarrow\mathrm{B}(Ind^G_T E)$ by \[\tilde{\phi}(f_A)(f_E)(g)\colon =\phi(f_A(g))(f_E(g))\]for all $g\in G,f_A\in Ind^G_T A,f_E\in Ind^G_T E$.

\begin{lemma}
$\tilde{\phi}$ is a well-defined $G$-$\ast$-homomorphism.
\end{lemma}

\begin{proof}
First of all, we need to check that it is well-defined:
\begin{eqnarray}
\tilde{\phi}(f_A)(f_E)(gt)&=&\phi(f_A(gt))(f_E(gt))\nonumber\\
                          &=&\phi(t^{-1}f_A(g))(t^{-1}f_E(g))\nonumber\\
                          &=&(t^{-1}\phi(f_A(g))t)(t^{-1}f_E(g))\nonumber\\
                          &=&t^{-1}\phi(f_A(g))(f_E(g))\nonumber\\
                          &=&t^{-1}\tilde{\phi}(f_A)(f_E)(g).\nonumber
\end{eqnarray}So $\tilde{\phi}(f_A)(f_E)\in Ind^G_T(E)$. And
\begin{eqnarray}
\parallel\tilde{\phi}(f_A)(f_E)(g)\parallel^2&=&\parallel\phi(f_A(g))(f_E(g))\parallel^2\nonumber\\
   &\leq&\parallel\phi(f_A(g))\parallel^2 \parallel f_E(g)\parallel^2\nonumber\\
   &\leq&\parallel\tilde{\phi}(f_A)\parallel^2 \parallel f_E\parallel^2.\nonumber
\end{eqnarray}Hence, $\tilde{\phi}(f_A)\in\mathrm{B}(Ind^G_T(E))$. It is straightforward to see that $\tilde{\phi}(f_A)^\ast$ exists and $\tilde{\phi}(f_A)^\ast\in\mathrm{B}(Ind^G_T(E))$. It is readily checked that $\tilde{\phi}$ is a $G$-$\ast$-homomorphism:
\begin{eqnarray}
(g\tilde{\phi}(f_A)g^{-1})(f_E)(x)&=&g\tilde{\phi}(f_A)(g^{-1}f_E)(x)\nonumber\\
   &=&g\phi(f_A(x))(g^{-1}f_E(x))\nonumber\\
   &=&g\phi(f_A(x))(f_E(gx))\nonumber\\
   &=&g\phi(gf_A(gx))(f_E(gx))\nonumber\\
   &=&g\tilde{\phi}(gf_A)(f_E)(gx)\nonumber\\
   &=&\tilde{\phi}(gf_A)(f_E)(x).\nonumber
\end{eqnarray}Hence, $(g\tilde{\phi}(f_A)g^{-1})(f_E)=\tilde{\phi}(gf_A)(f_E)$. 
\end{proof}

Let $F\in\mathrm{B}(E)$ where $E$ is a $T$-Hilbert $B$-module. We construct $\tilde{F}\in\mathrm{B}(Ind^G_T(E))$ as follows:\[\tilde{F}(f_E)(g):=F(f_E(g)).\]

\begin{lemma}
$\tilde{F}$ is a well-defined operator on the Hilbert $Ind^G_T B$-module map $E$. $\tilde{F}$ is $G$-invariant.
\end{lemma}

\begin{proof} 
\begin{eqnarray}
\tilde{F}(f_E)(gt)&=&F(f_E(gt))=F(t^{-1}f_E(g))\nonumber\\
    &=&t^{-1}F(f_E(g))t=t^{-1}.F(f_E(g))\nonumber\\
    &=&t^{-1}.\tilde{F}(f_E)(g).\nonumber
\end{eqnarray}So, $\tilde{F}(f_E)\in Ind^G_T E$.
\begin{eqnarray}
\tilde{F}(f_Ef_B)(g)&=&F(f_E f_B(g))=F(f_E(g) f_B(g))\nonumber\\
   &=&F(f_E(g))f_B(g)=\tilde{F}(f_E)(f_B)(g),\nonumber
\end{eqnarray}i.e. $\tilde{F}(f_E f_B)=\tilde{F}(f_E)f_B$. Hence, $\tilde{F}$ is an $Ind^G_T B$-module map. 
\begin{eqnarray}
\parallel\tilde{F}(f_E)\parallel_{Ind^G_T E}&=&\sup\parallel\tilde F(f_E)(g)\parallel=\sup\parallel F(f_E(g))\parallel\nonumber\\
  &\leq&\sup\parallel F\parallel\parallel f_E(g)\parallel\nonumber\\
  &=&\parallel F\parallel\sup\parallel f_E(g)\parallel\nonumber\\
  &=&\parallel F\parallel \parallel f_E\parallel.\nonumber
\end{eqnarray}So, $\tilde{F}\in\mathrm{B}(Ind^G_T E)$. Define $\tilde{F^\ast}(f_E)(g):=F^\ast(f_E(g))$.
\begin{eqnarray}
\langle\tilde{F}(f_E),f'_E\rangle(g)&=&\langle\tilde{F}(f_E)(g),f'_E(g)\rangle\nonumber \\ 
  &=&\langle F(f_E(g)),f'_E(g)\rangle\nonumber\\
  &=&\langle f_E(g),F^\ast (f'_E(g))\rangle\nonumber\\
  &=&\langle f_E,\tilde{F^\ast}(f'_E)\rangle(g).\nonumber
\end{eqnarray}So, $\tilde{F^\ast}=\tilde{F}^\ast$. $\tilde{F}$ is also $G$-continuous. i.e. $g\mapsto g.\tilde{F}$ is continuous in the norm topology.
\begin{eqnarray}
g.\tilde{F}(f_E)(x)&=&g\tilde{F}g^{-1}(f_E)(x)\nonumber\\
  &=&\tilde{F}(g^{-1}f_E)(g^{-1}x)\nonumber\\
  &=&F(g^{-1}f_E(g^{-1}x))\nonumber\\
  &=& F(f_E(x))\nonumber\\
  &=& \tilde{F}(f_E)(x).\nonumber
\end{eqnarray}So, $\tilde{F}$ is indeed $G$-invariant. 
\end{proof}

The \emph{induction homomorphism} \[ind^G_T\colon KK_T(A,B)\rightarrow KK_G(Ind^G_T(A),Ind^G_T(B))\] is defined by sending $x=[E,\phi,F]\in KK_T(A,B)$ to $ind^G_T(x)=[\tilde{E},\tilde{\phi},\tilde{F}]\in KK_G(Ind^G_T A,Ind^G_T B)$. It is clear that it is well-defined.

We give a proof of Theorem \ref{Wasserman1} now.

\vspace{5mm}

\emph{Proof of Theorem \ref{Wasserman1}:}
Let $x=[E,\phi,F]\in KK_T(A,B)$ and $i^\ast(ind^G_T(x))=[\tilde{E},\tilde{\phi}\circ i,\tilde{F}]$ where $\tilde{\phi}\circ i\colon A\rightarrow\mathrm{B}(\tilde{E})$. For $a\in A$, define $K_a\in Ind^G_T(A)$ by $K_a(g)=g^{-1}a$. Note that the $G$-action on $K_a$ gives $g.K_a=K_{ga}$. Under the isomorphism between $A\otimes C(G/T)$ and $Ind^G_T(A)$, we can identify $a\otimes 1\in A\otimes C(G/T)$ with $K_a\in Ind^G_T(A)$ for each $a\in A$.
\begin{eqnarray}
(\tilde{\phi}\circ i)(a)(f_E)(g)&=&\tilde{\phi}(K_a)(f_E)(g)\nonumber\\
  &=&\phi(K_a(g))(f_E(g))\nonumber\\
  &=&\phi(g^{-1}a)(f_E(g)).\nonumber
\end{eqnarray}And $res^G_T\circ i^\ast\circ ind^G_T(x)=[\tilde{E}\mid_T,(\tilde{\phi}\circ i)\mid_T,\tilde{F}\mid_T]$ 

For a $G$-$\ast$-homomorphism $f\colon B\rightarrow D$, the pushforward $f_\ast\colon KK_G(A,B)\rightarrow KK_G(A,D)$ is, by definition, $[M,\xi,N]\mapsto [M\otimes_f D,\xi\otimes id_D,N\otimes id_D]$ where $M\otimes_f D$ is the internal tensor product of the $G$-Hilbert $B$-module $M$ with $D$, viewed as a Hilbert $D$-module. For $x=[E,\phi,F]\in KK_T(A,B)$, we have
\[ev_\ast \circ res^G_T \circ i^\ast \circ ind^G_T(x)=[res^G_T(\tilde{E})\otimes_{ev} B,(res^G_T(\tilde{\phi}\circ i^\ast))\otimes id_B,res^G_T(\tilde{F})\otimes id_B]\]which is an element in $KK_T(A,B)$. $res^G_T(\tilde{E})$ is a $T$-Hilbert $B\otimes C(G/T)$-module, $res^G_T(\tilde{E})\otimes_{ev}B$ is then a $T$-Hilbert $B$-module, where $ev\colon B\otimes C(G/T)\rightarrow B$ is the evaluation at identity. For $f_E,f'_E\in res^G_T(\tilde{E})$, $b_1,b_2\in B$,
\begin{eqnarray}
\langle f_E\otimes b_1,f'_E\otimes b_2\rangle_{res^G_T(\tilde{E})\otimes_{ev}B}&=& b^\ast_1 ev(\langle f_E,f'_E\rangle)b_2\nonumber\\
  &=&b^\ast_1 \langle f_E,f'_E\rangle(1)b_2\nonumber\\
  &=&b^\ast_1 \langle f_E(1),f'_E(1)\rangle b_2\nonumber\\
  &=&\langle f_E(1)b_1,f'_E(1)b_2\rangle.\nonumber
\end{eqnarray}Our goal is to show that $x=ev_\ast\circ res^G_T\circ i^\ast\circ \iota^G_T(x)\in KK_T(A,B)$.

\vspace{3mm}

\emph{Claim}: $\tilde{E}\otimes_{ev}B$ is isomorphic to $E$ as $T$-Hilbert $B$-modules, i.e. $res^G_T(\tilde{E})\otimes_{ev}B\cong E$.

\emph{Proof of claim}: Define $Q\colon res^G_T(\tilde{E})\otimes_{ev}B\rightarrow E$ by $f_E\otimes b\mapsto f_E(1)b$. 
\begin{eqnarray}
Q((f_E\otimes b)b_1)&=&Q(f_E\otimes bb_1)=f_E(1)bb_1=(f_E(1)b)b_1\nonumber\\
                    &=&Q(f_E\otimes b)b_1,\nonumber
\end{eqnarray}
\begin{eqnarray}
Q(t(f_E\otimes b))&=&Q(tf_E\otimes tb)=(tf_E(1))(t(b))=t(f_E(1)b)\nonumber\\
                  &=&tQ(f_E\otimes b).\nonumber
\end{eqnarray}Hence, $Q$ is a $T$-Hilbert $B$-module map. Since $G$ is compact, for each $x\in E$, we can choose a constant function $f_x\colon G\rightarrow E$ in $\tilde{E}$ defined by $f_x(g)=x$ for all $g\in G$. Then $Q(f_x\otimes b)=xb$ for all $b\in B$. So $Q$ is surjective. Notice that \[\langle f'_E\otimes b_1,f''_E\otimes b_2\rangle=\langle f'_E(1)b_1,f''_E(1)b_2\rangle\] \[\langle Q(f'_E\otimes b_1),Q(f''_E\otimes b_2)\rangle=\langle f'_E(1)b_1,f''_E(1)b_2\rangle.\]So, $Q$ is isometric. Hence $Q$ is an isomorphism between $\tilde{E}\otimes_{ev}B$ and $E$ as $T$-Hilbert $B$-modules.

\vspace{3mm}

\emph{Claim}: For any $a\in A$, $b\in B$, the following diagram is commutative:
\[\begin{CD}
res^G_T(\tilde{E})\otimes_{ev}B@>{(res^G_T(\tilde{\phi}\circ i)\otimes id_B)(a\otimes b)}>>res^G_T(\tilde{E})\otimes_{ev}B\\
@VV {Q} V    @VV {Q} V\\
E@> {\phi(a)}>> E\\
\end{CD}\]

\emph{Proof of claim}: For any $f_E\otimes b\in res^G_T(\tilde{E})\otimes_{ev}B$, 
\begin{eqnarray}
Q((res^G_T(\tilde{\phi}\circ i)\otimes id_B)(a \otimes b)(f_E\otimes b))&=&Q(res^G_T(\tilde{\phi}\circ i)(a)(f_E)\otimes id_B(b))\nonumber\\
   &=&(res^G_T(\tilde{\phi}\circ i)(a))(f_E)(1)b\nonumber\\
   &=&\phi(K_a(1))(f_E(1))b\nonumber\\
   &=&\phi(a)(f_E(1))b.\nonumber
\end{eqnarray}
And
\begin{eqnarray}
\phi(a)(Q(f_E\otimes b))=\phi(a)(f_E(1)b)=\phi(a)(f_E(1))b.\nonumber
\end{eqnarray}So the claim is proved.

\vspace{3mm}

\emph{Claim}: The following diagram is commutative:
\[\begin{CD}
res^G_T(\tilde{E})\otimes_{ev}B@>{res^G_T(\tilde{F})\otimes id_B}>>res^G_T(\tilde{E})\otimes_{ev}B\\
@VV {Q} V    @VV {Q} V\\
E@> {F}>> E\\
\end{CD}\]

\emph{Proof of claim}: For any $f_E\otimes b\in res^G_T(\tilde{E})\otimes_{ev}B$, 
\begin{eqnarray}
Q((res^G_T\tilde{F})\otimes id_B)(f_E\otimes b)&=&Q(res^G_T\tilde{F}(f_E)\otimes id_B (b))\nonumber\\
   &=&\tilde{F}(f_E)(1)b\nonumber\\
   &=&F(f_E(1))b.\nonumber
\end{eqnarray}And
\begin{eqnarray}
F(Q(f_E\otimes b))=F(f_E(1)b)=F(f_E(1))b.\nonumber
\end{eqnarray}The claim is proved. We have shown that $x=ev_\ast \circ res^G_T\circ i^\ast \circ ind^G_T(x)\in KK_T(A,B)$. 

\vspace{3mm}

On the other hand, take any $y=[V,\psi,W]\in KK_G(A,B\otimes C(G/T))$. By Prop.20.2.4 in \cite{Blackadar}, we can assume that $W$ is $G$-invariant. $V$ is a $G$-Hilbert $B\otimes C(G/T)$-module. $Ind^G_T(res^G_T(V)\otimes_{ev}B)$ is a $G$-Hilbert $B\otimes C(G/T)$-module.

\vspace{3mm}

\emph{Claim}: $V$ is isomorphic to $Ind^G_T(res^G_T(V)\otimes_{ev}B)$ as $G$-Hilbert $B\otimes C(G/T)$-modules. 

\emph{Proof of claim}: Define $\Phi\colon V \rightarrow Ind^G_T(res^G_T(V)\otimes_{ev}B)$ by $\Phi(ef_B)(g)=g^{-1}e\otimes f_B(g)$ for any $e\in V,f_B\in Ind^G_T(B)\cong B\otimes C(G/T),g\in G$. Then
\begin{eqnarray}
\parallel\Phi(ef_B)\parallel^2&=& \max\parallel\Phi(ef_B)(g)\parallel^2\nonumber\\
  &=&\max\parallel g^{-1}e\otimes f_B(g)\parallel^2\nonumber\\
  &=&\max\parallel\langle g^{-1}e\otimes f_B(g),g^{-1}e\otimes f_B(g)\rangle\parallel\nonumber\\
  &=&\max\parallel f_B (g)^\ast\langle g^{-1}e,g^{-1}e\rangle(1) f_B(g)\parallel\nonumber\\
  &=&\max\parallel f_B (g)^\ast g^{-1}\langle e,e\rangle(1) f_B(g)\parallel,\nonumber
\end{eqnarray}
\begin{eqnarray}
\parallel ef_B\parallel^2&=& \max\parallel ef_B(g)\parallel^2\nonumber\\
   &=&\max\parallel f_B(g)^\ast \langle e,e\rangle(g) f_B(g)\parallel\nonumber\\
   &=&\max\parallel f_B(g)^\ast g^{-1}\langle e,e\rangle(1) f_B(g)\parallel.\nonumber
\end{eqnarray}So, $\Phi$ preserves the norm.
\begin{eqnarray}
\Phi(ef_Bf'_B)(g)&=&g^{-1}e\otimes f_B(g)f'_B(g)=(g^{-1}e\otimes f_B(g))f'_B(g)=\Phi(ef_B)(g)f'_B(g)\nonumber\\
  &=&(\Phi(ef_B)f'_B)(g),\nonumber
\end{eqnarray}
\begin{eqnarray}
g\Phi(ef_B)(g_1)&=&\Phi(ef_B)(g^{-1}g_1)\nonumber\\
   &=&(g^{-1}g_1)^{-1}e\otimes f_B(g^{-1}g_1)\nonumber\\
   &=&g^{-1}_1 ge\otimes gf_B(g_1)\nonumber\\
   &=&\Phi((ge)(gf_B))(g_1)\nonumber\\
   &=&\Phi(g(ef_B))(g_1).\nonumber
\end{eqnarray}So, $\Phi$ is a $G$-Hilbert $B\otimes C(G/T)$-module map. And it is clear that $\Phi$ is surjective so it defines an isomorphism between $Ind^G_T(res^G_T (V)\otimes_{ev}B)$ and $V$ as $G$-Hilbert $B\otimes C(G/T)$ modules.

\vspace{3mm}

\emph{Claim}: For any $a\in A$, the following diagram is commutative:
\[\begin{CD}
V@>{\psi(a)}>>V\\
@VV {\Phi} V    @VV {\Phi} V\\
Ind^G_T(res^G_T V\otimes_{ev}B)@> {Ind^G_T(res^G_T\psi\otimes Id_B)\circ i(a)}>> Ind^G_T(res^G_T V\otimes_{ev}B)\\
\end{CD}\]

\emph{Proof of claim}: For any $e\in V$, $f_B\in Ind^G_T(B)\cong B\otimes C(G/T)$, $g\in G$,
\begin{eqnarray}
\Phi(\psi(a)(ef_B))(g)&=&\Phi((\psi(a)(e))f_B)(g)\nonumber\\
  &=&g^{-1}(\psi(a)(e))\otimes f_B(g)\nonumber\\
  &=&\psi(g^{-1}a)(g^{-1}e)\otimes f_B(g).\nonumber
\end{eqnarray}The last equality is due to: 
\begin{eqnarray}
\psi(g^{-1}a)(g^{-1}e)&=&g^{-1}\psi(a)gg^{-1}e=g^{-1}\psi(a)(e).\nonumber
\end{eqnarray}
On the other hand,
\begin{eqnarray}
(Ind^G_T(res^G_T\psi\otimes Id_B)\circ i(a))(\Phi(ef_B))(g)&=& (Ind^G_T(res^G_T\psi\otimes Id_B))(K_a)(\Phi(ef_B))(g) \nonumber\\
  &=& (res^G_T\psi\otimes Id_B)(K_a(g))(\Phi(ef_B)(g))\nonumber\\
  &=& (res^G_T\psi\otimes Id_B)(g^{-1}a)(g^{-1}e\otimes f_B(g))\nonumber\\
  &=& \psi(g^{-1}a)(g^{-1}e)\otimes f_B(g).\nonumber
\end{eqnarray}It proves the claim.

\vspace{3mm}

\emph{Claim}: The following diagram is commutative:
\[\begin{CD}
V@>{W}>>V\\
@VV {\Phi} V    @VV {\Phi} V\\
Ind^G_T(res^G_T V\otimes_{ev}B)@> {Ind^G_T(res^G_T W\otimes Id_B)}>> Ind^G_T(res^G_T V\otimes_{ev}B)\\
\end{CD}\]

\emph{Proof of claim}: For any $v\in V$, $f_B\in Ind^G_T(B)$, $g\in G$,
\begin{eqnarray}
Ind^G_T(res^G_T W\otimes Id_B)(\Phi(vf_B))(g)&=& (res^G_T W\otimes Id_B)(\Phi(vf_B)(g))\nonumber\\
   &=& (res^G_T W\otimes Id_B)(g^{-1}v\otimes f_B(g))\nonumber\\
   &=& W(g^{-1}v)\otimes f_B(g),\nonumber
\end{eqnarray}
\begin{eqnarray}
\Phi\circ W(vf_B)(g)&=& \Phi(W(vf_B))(g)=\Phi(W(v)f_B)(g)\nonumber\\
   &=& g^{-1}(W(v))\otimes f_B(g).\nonumber
\end{eqnarray}Since $W$ is $G$-invariant, then
\begin{eqnarray}
g^{-1}(W(v))&=& g^{-1}(W(gg^{-1}v))=g^{-1}.W(g^{-1}v)=W(g^{-1}v).\nonumber
\end{eqnarray}The last equality is by the $G$-invariance of $W$. Hence we have shown that $y=i^\ast\circ ind^G_T\circ ev_\ast \circ res^G_T(y)\in KK_G(A,B\otimes C(G/T))$. It concludes our proof of the theorem.

\section{Proof of Theorem  \ref{Rosenberg_result}}\label{Rosenbergsection}

In this section, we give a sketch proof of Theorem \ref{Rosenberg_result}:

\vspace{3mm}

\begin{proof}
The basic idea is similar to the one proved by Rosenberg and Schochet in Theorem 3.7 (i) of \cite{RosenbergSchochet} for the case of $K$-theory of $\mathrm{C}^\ast$-algebras. Therefore we content ourselves here with a sketch of the proof. A special case of Theorem 4.10 in \cite{Kasparov2} showed that there is a Poincar\'{e} duality\[\delta\colon KK_G(C(G/T),\mathbb{C})\rightarrow KK_G(\mathbb{C},C(G/T))\]which is an isomorphism. And more generally, we have an isomorphism \[\delta_{C(G/T)}\colon KK_G (C(G/T),C(G/T))\rightarrow KK_G (\mathbb{C},C(G/T)\otimes C(G/T)).\]By a theorem of McLeod \cite{McLeod}, \[KK_G(\mathbb{C},C(G/T)\otimes C(G/T))\cong K_G^\ast(G/T \times G/T)\cong K_T^\ast(G/T)\cong R(T)\otimes_{R(G)}R(T).\]Steinberg's theorem \cite{Steinberg} provides a free basis $\{e_\omega\}_{\omega\in W}$ for $R(T)$ as a $R(G)$-module, where $W\cong N(T)/T$ is the Weyl group of $(G,T)$. Then there exist an unique set of elements $\{b_\omega\}_{\omega\in W}$ of $R(T)\cong KK_G(\mathbb{C},C(G/T))$ such that \[\delta_{C(G/T)}(1_{C(G/T)})=\sum_{\omega\in W}b_\omega\otimes_\mathbb{C}e_\omega.\]Note that $\otimes_\mathbb{C}$ is the Kasparov product. For $\omega\in W$, let \[a_\omega=\delta^{-1}(b_\omega).\]Then we have, for $1_{C(G/T)}\in KK_G(C(G/T),C(G/T))$, 
\begin{eqnarray}
1_{C(G/T)}&=&\delta_{C(G/T)}^{-1}(\delta_{C(G/T)}(1_{C(G/T)}))\nonumber\\
   &=&\delta_{C(G/T)}^{-1}(\sum_{\omega\in W}b_\omega\otimes_{\mathbb{C}}e_\omega)\nonumber\\
   &=&\sum_{\omega\in W}\delta^{-1}(b_\omega)\otimes_{\mathbb{C}}e_\omega\nonumber\\
   &=&\sum_{\omega\in W}a_\omega\otimes_\mathbb{C} e_\omega.\nonumber
\end{eqnarray}The third equality is done by associativity of Kasparov product. Then we have the following calculation for any $v\in W$: 
\begin{eqnarray}
e_v&=& e_v\otimes_{C(G/T)}1_{C(G/T)}=e_v\otimes_{C(G/T)}(\sum_{\omega\in W}a_\omega\otimes_\mathbbm{C}e_\omega)\nonumber\\
   &=& \sum_{\omega\in W}(e_v\otimes_{C(G/T)}a_\omega)\otimes_\mathbb{C}e_\omega\nonumber
\end{eqnarray}which means that if $v=\omega$, $
e_v\otimes_{C(G/T)}a_\omega=1_{R(G)}$. And $e_v\otimes_{C(G/T)}a_\omega=0$ otherwise. For any element $y\in KK_T(A,B)\cong KK_G(A,B\otimes C(G/T))$ (the isomorphism is by Theorem \ref{Wasserman1}), 
\begin{eqnarray}
y &=& y\otimes_{C(G/T)}1_{C(G/T)}\nonumber\\
  &=& y\otimes_{C(G/T)}(\sum_{\omega\in W}a_\omega\otimes_\mathbb{C}e_\omega)\nonumber\\
  &=& \sum_{\omega\in W}(y\otimes_{C(G/T)}a_\omega)\otimes_\mathbb{C}e_\omega. \label{equation1}
\end{eqnarray}Note that $y\otimes_{C(G/T)}a_\omega\in KK_G(A,B)$. If \[y=\sum_{\omega\in W}x_\omega\otimes_\mathbbm{C}e_\omega\]for some $x_\omega\in KK_G(A,B)$, then
\begin{eqnarray}
y\otimes_{C(G/T)}a_u&=&(\sum_{\omega\in W}x_\omega\otimes_\mathbb{C}e_\omega)\otimes_{C(G/T)}a_u\nonumber\\
  &=&\sum_{\omega\in W}x_\omega\otimes_\mathbb{C}(e_\omega\otimes_{C(G/T)}a_u)\nonumber\\
  &=&\sum_{\omega\in W}x_\omega\otimes_\mathbb{C}\delta_{uw}\nonumber\\
  &=&x_u.\nonumber
\end{eqnarray}Hence, equation (\ref{equation1}) is an unique expression for $y\in KK_T(A,B)$. It means that $KK_T(A,B)$ and $R(T)\otimes_{R(G)}KK_G(A,B)$ are isomorphic as $R(G)$-modules. It is clear that they are also isomorphic as $R(T)$-modules.
\end{proof}

\section*{Acknowledgements}
The author would like to thank Reyer Sjamaar for all his encouragements and his suggestion of this research topic. The author would also like to thank the referee for useful suggestions.

\vspace{3mm}

\small{School of Liberal Arts and Sciences, Canadian University of Dubai, UAE}

\emph{Email: leung@cud.ac.ae}

\end{document}